\documentclass{amsart}
\usepackage{graphicx}
\usepackage{hyperref}
\usepackage{graphicx,amssymb,amstext,amsthm,url}
\newcommand{\Deltab}{\bigtriangleup}
\usepackage{color}

\begin{document}

\newtheorem{theorem}{Theorem}  [section]
\newtheorem{proposition}[theorem]{Proposition}
\newtheorem{lemma}[theorem]{Lemma}
\newtheorem{clm}[theorem]{Claim}
\newtheorem{corollary}[theorem]{Corollary}
\newtheorem{conj}[theorem]{Conjecture}
\newtheorem{conjecture}[theorem]{Conjecture}

\title{Generate Delta matroids from matroids}
%%

%\author[R.\ Avohou]{Remi Cocou Avohou}
%\address{}
%\email{}
%\author[B.\ Servatius]{Brigitte Servatius}
%\address{Mathematical Sciences, Worcester Polytechnic Institute,
%Worcester MA 01609-2280.}
%\email{bservat@wpi.edu}
%\author[H.\ Servatius]{Herman Servatius}
%\address{Mathematical Sciences, Worcester Polytechnic Institute,
%Worcester MA 01609-2280.    }
%\email{hservat@wpi.edu}

\author{R\'emi Cocou Avohou}
\address[R.C.A.]{
Max Planck Institut f\"ur Mathematik, Vivatsgasse 7, 53111 Bonn, Germany, \& ICMPA-UNESCO Chair, 072BP50, Cotonou,
Rep. of Benin, \& Ecole Normale Superieure, B.P 72, Natitingou, Benin}
\email{avohou.r.cocou@mpim-bonn.mpg.de}

\author{Brigitte Servatius}
\address[B.S.]{Mathematical Sciences, Worcester Polytechnic Institute, Worcester MA 01609-
2280}
\email{bservat@wpi.edu }

\author{Herman Servatius}
\address[B.S.]{Mathematical Sciences, Worcester Polytechnic Institute, Worcester MA 01609-
2280}
\email{hservat@wpi.edu }

\maketitle

\begin{abstract}
We give necessary and sufficient conditions for two matroids on
the same ground set to be the upper and lower matroid of a $\Delta$-matroid.
\end{abstract}

\section{Matroids and $\Delta$-matroids}
A {\em matroid} $M$ is a finite set $E$ and a collection $\mathcal{B}$ of subsets of $E$
satisfying the condition that if
\begin{enumerate}\renewcommand{\theenumi}{MB}
\item
If $B_1$ and $B_2$ are in $\mathcal{B}$ and $x \in B_1 \setminus B_2$ then there
exists a $y \in  B_2 \setminus B_1$ such that $(B_1 \cup \{ y \} ) \setminus \{ x \} = B_1 \Deltab \{x,y\} \in \mathcal{B}$
\label{BaseAxiom}
\end{enumerate}
Axiom (MB) is called the
%matroid {\em base axiom} or
{\em basis exchange axiom}. Sets in $\mathcal{B}$ are called {\em bases} of $M$. Subsets of bases are called {\em independent sets}, sets which are not independent are called {\em dependent}, minimal dependent sets are called cycles, and sets containing a basis are called {\em spanning}.

Matroids were introduced by Whitney \cite{whitney} in 1935. Since then many texts on matroid theory have appeared; see \cite{tutte}, or the standard text for the graph theorist \cite{welsh}, the standard resource for the geometer/algebraist \cite{oxley}, an applied approach \cite{andrec}, or
a new text \cite{leopit}.

Replacing the set difference in Axiom~(MB) by the symmetric difference we obtain the symmetric exchange axiom ($\Delta$F) used by
Bouchet~\cite{bouchet2} to define $\Delta$-matroids. A more recent combinatorics formulation introduced in \cite{avoserv1} may be of interest to an interested reader.

A {\em $\Delta$-matroid} $D$ is a finite set $E$ and a collection $\mathcal{F}$ of subsets of $E$
satisfying the condition that if
\begin{enumerate}\renewcommand{\theenumi}{$\Delta$F}
\item If $F_1$ and $F_2$ are in $\mathcal{F}$ and $x \in F_1 \Deltab  F_2$ then there
exists a $y \in  F_2 \Deltab F_1$ such that $F_1 \Deltab \{x,y\} \in \mathcal{F}$.
\label{FeaseAxiom}
\end{enumerate}
Axiom ($\Delta$F) is called the {\em symmetric exchange axiom} and the sets in $\mathcal{F}$ are called the {\em feasible
sets} of $D$.  It is important to note that $y$ may equal $x$, so
$|F_1 \Deltab \{x,y\}| - |F_1| \in \{0, \pm 1 , \pm 2\}$.

$\Delta$-matroids were independently introduced by Dress and Havel \cite{andrea} as matroids, and by Chandasekaran and Kabadi \cite{chandra} as pseudometroids. Bouchet developed the properties of $\Delta$-matroids and related structures in a series of papers on multimatroids \cite{bouchet5, bouchet6, bouchet7, bouchet8}.
%{bouchet5}{chandra}{andrea}

%As a first example, let $E$ be a set on $n$ elements and let $ \mathcal{F}$  consist of all subsets of $E$
%that have either $k$ or $k+2$ elements, $\mathcal{B}_k$ consist of $k$ element subsets and $\mathcal{B}_{k+2}$ consist of all
%$k+2$ element subsets of $S$. Clearly $\mathcal{B}_k$ and $\mathcal{B}_{k+2}$ are base sets of matroids, they are called
%uniform matroids $U_{k,n}$ resp. $U_{k+2,n}$, but $\mathcal{F}= \mathcal{B}_k \cup \mathcal{B}_{k+2}$ is not, as Axiom (MB)
%is not satisfied for $B_1$ a $k+2$ element set and $B_2$ a $k$-element subset of $B_1$. However, $\mathcal{F}$ satisfies ($\Delta$F).

It was observed by Bouchet that the bases of every matroid are the feasible sets of $\Delta$-matroid,
and, since (MB) forces all bases of $M$ to
be equicardinal, which in some literature is listed as an axiom, not all $\Delta$-matroids arise in this fashion.
He also noted that there are two obvious matroids associated with every $\Delta$-matroid;
$M_u$, the {\em upper matroid}, whose bases are the feasible sets with largest cardinality, and
$M_l$, the {\em lower matroid}, whose bases are the feasible sets with least cardinality,~\cite{Boucher3}.
It is the point of view of this paper study the $\Delta$-matroid specifically with regard to the
relation it bears to its upper and lower matroids.

If you regard the bases of a matroid to be feasible, then the
upper and lower matroids of the resulting $\Delta$-matroid are the same.
For an example where they are distinct, let $E$ be a set on $n$ elements, $0 < k < n$, and let $ \mathcal{F}$
consist of all subsets of $E$ that have either $k-1$ or $k+1$ elements. Then it is easy to check that
$\mathcal{F}$ is the collection of feasible sets of a $\Delta$-matroid on $E$.  The upper matroid $M_u$ is the uniform matroid
$U_{k+1,n}$, and lower matroid is $U_{k-1,n}$.
However, we also have the following
extreme examples.

\begin{theorem}\label{indy}
Let $M$ be a matroid and let $\mathcal{I}$ be the collection of independent
sets of $M$.  Then $\mathcal{I}$ satisfies the symmetric exchange axiom.
\end{theorem}

\begin{proof}
   Let $I_1, I_2 \in \mathcal{I}$ and let $x \in I_1 \bigtriangleup I_2$.

  If $x \in I_1\setminus I_2$, then $I_1 \bigtriangleup \{x,x\} = I_1 - x$ is independent
as required, so assume that $x \in I_2\setminus I_1$.  If $I_1 + x$ is independent, then set
$y = x$, so that  $I_1 \bigtriangleup \{x,y\} = I_1 + x$ is independent.  If, on the other hand,
$I_1 + x$ is dependent, then $I_1 + x$ contains a unique cycle containing $x$, but which is not
contained in $I_2$, so there exists $y \in I_1 \setminus I_2$, so that $I_1 \bigtriangleup \{x,y\}= I_1 - x + y$
is independent.
\end{proof}

So $\mathcal{I}$ are the feasible sets of a $\Delta$-matroid with upper-matroid $M$ and lower matroid with only
the empty basis.  And, not surprisingly, the dual result also holds.

\begin{theorem}\label{spanny}
Let $M$ be a matroid and let $\mathcal{S}$ be the collection of spanning
sets of $M$.  Then $\mathcal{S}$ satisfies the symmetric exchange axiom.
\end{theorem}

\begin{proof}
   Let $S_1, S_2 \in \mathcal{S}$ and let $x \in S_1 \bigtriangleup S_2$.

  If $x \in S_2\setminus S_1$, then $S_1 \bigtriangleup \{x,x\} = S_1 + x$ is spanning
as required, so assume that $x \in S_1 \setminus S_2$.  If $S_1 - x$ is spanning, then set
$y = x$, so that  $S_1 \bigtriangleup \{x,x\} = S_1 - x$ is spanning.  If, on the other hand,
$S_1 - x$ is not spanning, then since $S_2$ is spanning it contains an element $y \not \in S_1$ so that
$S_1 - x + y = S_1 \bigtriangleup \{x,y\}$ spans.
\end{proof}

%So there are three collections of feasible sets arising naturally from any matroid.
So any matroid $M$ may be naturally viewed as a $\Delta$-matroid in three different ways, namely by considering $\mathcal{F} = \mathcal{B}$ where $M_u = M_l = M$, or $\mathcal{F} = \mathcal{I}$, where $M_u = M$ and the lower matroid has rank zero, or $\mathcal{F}$ equals the collection of spanning sets of $M$ where $M_l = M$ and the upper matroid has rank $|E|$.

The feasible sets of cardinality between the ranks of $M_l$ and $M_u$ can also be related to the upper and lower
matroids.

\begin{theorem} \label{uplow}
If $D$ is a  $\Delta$-matroid and $F$ is a feasible set of $D$, then $F$ is spanning in the lower matroid
and independent in the upper matroid.
\end{theorem}

\begin{proof}
Let $F$ be a feasible set and let $L$ be the basis of $M_l$ with the largest possible intersection with $F$.
If  $L \nsubseteq  F$, there exists $x\in L \setminus F$. By the symmetric exchange axiom, there must be a $y\in L \Deltab F$
such that $L \Deltab \{ x,y \} $ is feasible. If $y \in L \setminus F$, then $|L \Deltab \{ x,y \}| < |L|$, which is impossible, so $y \in F\setminus L$ and
 $L \Deltab \{ x,y \}$ is a basis of $M_l$ which intersects $U$ in more elements than $L$ does, a contradiction.
Therefore $L \subseteq  F$, that is, $F$ is spanning in $M_l$.

For independence, let $F$ be a feasible set and  let $U$ be a basis of $M_u$ with largest possible intersection with $F$.
If $|F| \nsubseteq  |U|$,  there exists $x \in F \setminus U$.
By the symmetric exchange axiom, there must be a $y\in F \Deltab U$
such that $U \Deltab \{ x,y \} $ is feasible. If $y \in F \setminus U$, then $|U \Deltab \{ x,y \}| > |U|$, which is impossible,
so $y \in U\setminus F$ and $U \Deltab \{ x,y \}$ is a basis of $M_u$ which intersects $F$ in more elements than $U$ does, a contradiction.
Therefore $F$ is contained in some basis of $M_u$ and is independent in $M_u$.
%\footnote{This is necessary to do, but it s a virtual rewrite, so in the final
%paper we could just say that the argument is analogous.}
\end{proof}

In particular, Theorem~\ref{uplow} leads to a necessary condition on two matroids of different ranks on a set $E$ to be
the upper and lower matroids of a $\Delta$-matroid.
\begin{corollary} \label{uplow01}
If $D$ is a  $\Delta$-matroid on a set $E$ with upper matroid $M_u$ and lower matroid $M_l$, then
every basis of $M_u$ is a spanning set of $M_l$ and every basis of $M_l$ is independent in $M_u$.
\end{corollary}

The converse of Corollary~\ref{uplow01} is the topic of the next section.

\section{Characterizing upper and lower matroids}
If we wanted to study a $\Delta$-matroid $D$, and were given its upper and lower matroids, $M_u$ and $M_l$, it would be convenient if we could simply
compute the rest of the feasible sets from these extreme classes, but that is impossible, as the simple feasible
collections $\{ \{a,b\}, \{a\}, \{b\}, \emptyset \}$ and
$\{ \{a,b\}, \emptyset \}$ demonstrate, since they have the same upper and lower matroids.

In general, if the rank functions satisfy $\mbox{rk}(M_u) - \mbox{rk}(M_u) > 2$, it is impossible for there to be no feasible sets of intermediate
cardinality, since symmetric exchange will force the existence of feasible sets both of cardinality one or two greater than minimum,
and of cardinality one or two less than maximum.  It is more likely to find a general construction by declaring as many of the sets of intermediate
cardinality to be feasible as possible, within the restrictions of Theorem~\ref{uplow}.  This construction works
in the examples of
Theorems~\ref{indy} and~\ref{spanny} as well as in the following theorem relating the cycle matroid of a simple graph $G = (V, E)$ on the ground set $E$ and its ($2$-dimensional generic) rigidity matroid, where a set $F\subseteq E$ of edges is independent if $|F'|\leq 2|V (F')| - 3$ holds for all non-empty subsets $F' \subseteq F$ . The set $V (F)$ here denotes set of endpoints of edges in $F$ . Edge sets violating this inequality are called overbraced. For an introduction to combinatorial rigidity see \cite{servatius}.

\begin{theorem}
Let $G = (V,E)$ be a connected graph and consider the connectivity- or cycle-matroid $M_c$ and the $2$-dimensional
generic rigidity matroid $M_r$ on the edge set $E$ of $G$.
Let $\mathcal{F}$ denote the collection of edge sets which induce graphs containing a spanning tree, but are not over-braced.
Then $\mathcal{F}$ satisfies the symmetric exchange property.
\end{theorem}

\begin{proof}
Let $F, F' \in \mathcal{F}$ and let $x \in F \setminus F'$.  If $(V,F - x)$ is connected, then choose $y = x$, and
$F \bigtriangleup \{x,x\} = F \setminus \{x\}$ contains a spanning tree and is not over-braced.
Otherwise, $F-x$ is disconnected, and $x$ is a bridge of $(V,F)$.  Since $(V,F')$ contains a spanning tree, and
does not contain $x$, it contains an edge $y$ so that every spanning tree of $(V,F-x+y)$ contains $y$.
Moreover,  $G(V,F-x+y)$ cannot be over-braced, since  $G(V,F-x)$ is not over-braced and no rigidity cycle can
contain a bridge.

Now we must consider the case when $x \in F' \setminus F$.  If $(V,F + x)$ is not over-braced, we can choose $y = x$, and
$F \bigtriangleup \{x,y\} = F + \{x\}$ has the desired property.  Otherwise, $(V,F+x)$ contains a unique rigidity cycle, which is necessarily
edge $2$-connected.  Since the set of edges of that rigidity cycle is not contained in $F'$, there is an edge $y \in F-F'$ whose deletion
does not disconnect it, and so does not increase the the number of connected components of $(V,F + y)$, but whose deletion does leave
$(V,F+y)$ no longer over-braced.  So $F+y-x$ has the desired property.
\end{proof}

This proof depended on the fact that the graph of a rigidity cycle must be a $2$-connected graph, or, equivalently, that the
cycle matroid on those edges is connected.  This is not true in general, even if $M_u$ and $M_l$ are
the upper and lower matroids of a $\Delta$-matroid, as the
following example shows.
%
%
%\begin{conjecture}
%Let $M_u$ and $M_l$ be matroids on the same set $E$ such that every basis of $M_l$ is independent in $M_u$,
%and such that every basis of $M_u$ is spanning in $M_l$.
%Then setting $$\mathcal{F} = \{F \subseteq E \mid \mbox{$F$ is spanning in $M_u$} \wedge \mbox{$F$ is spanning in $M_u$}\} $$
%we have that $\mathcal{D}$ is the set of feasible sets of a $\Delta$-matroid.
%\end{conjecture}
%
%To do this, we considered:
%\begin{conjecture}
%Let $M_u$ and $M_l$ be matroids such that every basis of $M_l$ is independent in $M_u$, and every basis of $M_u$ is spanning in $M_l$.
%
%Then every cycle in $M_u$ is connected, in a matroid sense, in $M_l$.
%\end{conjecture}
%
%This is false.
Consider $E = \{1,2,3,a,b,c\}$, $M_u = U_{5,6}(E)$,
$M_l  = U_{2,3}(\{1,2,3\}) \oplus U_{2,3}(\{a,b,c\})$.
Then it is easy to check that $M_u$, $M_l$ are the upper and lower matroids of a delta-matroid $D$, with feasible sets
$\{B_1 \cup B_2 \mid B_1 \in \mathcal{B}(M_u), B_2 \in \mathcal{B}(M_l)\}$.  Here the upper matroid is a cycle,
and the lower matroid is disconnected.
%To check, we need only compare bases of different types.
%W.l.o.g.\ if $E_1 = E - \{1,a\}$, $E_2 = E - \{1\}$, then $E_1 \Delta E_2 = \{a\}$, and for $x = a$, take
%$y = a$, which swaps $a$ in and out between feasible sets.
%
%If $E_1 = E - \{1,a\}$, $E_2 = E - \{b\}$, then $E_1 \Delta E_2 = \{a,b,1\}$ and we have more cases to check.
%
%$E_1$, $x=1$, $y=1$, $E_1 - \{a\}$, a lower basis, feasible.
%
%$E_1$, $x=a$, $y=a$, $E_1 - \{1\}$, a lower basis, feasible.
%
%$E_1$, $x=b$, $y=a$, $E_1 - \{b\} + \{a\}$, an upper basis, feasible.
%
%$E_2$, $x=1$, $y=1$, $E_2 - \{a,1\}$, a upper basis, feasible.
%
%$E_2$, $x=a$, $y=b$, $E_2 - \{a\} + \{b\}$, a lower basis, feasible.
%
%$E_2$, $x=b$, $y=a$, $E_2 - \{a\} + \{b\}$, a lower basis, feasible.
%

We would like to have that every cycle in $M_u$ is a union of cycles in $M_l$, which is not necessarily the
case even if
$M_u$ and $M_l$ are matroids such that every basis of $M_l$ is independent in $M_u$, and every basis of $M_u$ is spanning in $M_l$,
as the following example demonstrates.
Consider the cycle matroids of multigraphs of Figure~\ref{county01}.
\begin{figure}[htb]
   \includegraphics[scale=.6]{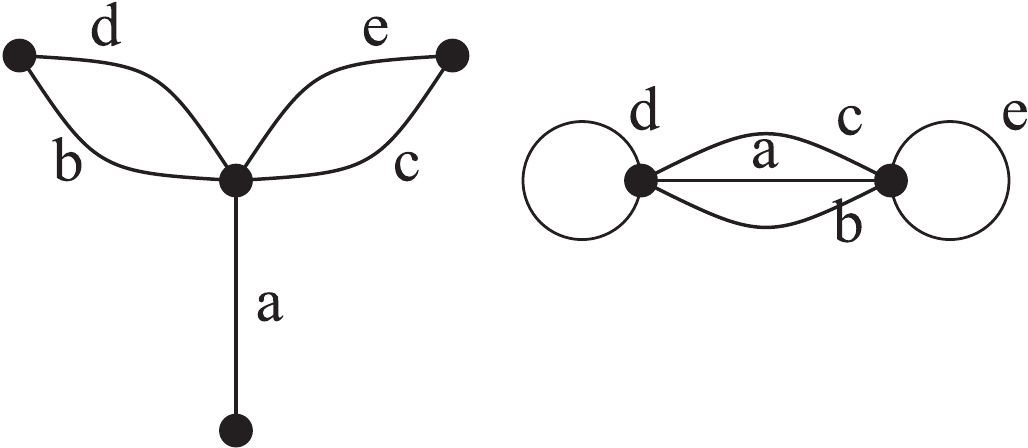}
   \caption{Two graphs on the same edge  set.\label{county01}}
\end{figure}
Every basis of $M_u$, the matroid for the graph on the left, contains $a$ and so spans the matroid $M_l$, for the
graph on the right;  and each of the bases of $M_l$ have one element and so are independent in $M_u$, which has no loops.
However the cycle $\{d,b\}$ of $M_u$ is not a union of cycles of $M_l$.
Moreover, $M_u$ and $M_l$ cannot be the
upper and lower matroids of any $\Delta$-matroid.  Consider $\{a,d,e\}$ and $\{b\}$.  Then
$a \in \{a,d,e\} \bigtriangleup \{b\}$, and, if there was a $\Delta$-matroid, there would exist a
$y \in \{a,d,e\} \bigtriangleup \{b\} = \{a,b,d,e\}$ so that $\{a,d,e\} \Delta \{a,y\}$ is feasible.
But $y \not \in  \{a,d,e\}$ since $\{d,e\}$ is not lower spanning, and $y \neq b$ since $\{b,d,e\}$ is not
upper-independent.  So no choice of $y$ could give a feasible set.

But we do have the following result, which gives us a new necessary condition
and will allow us to characterize upper and lower matroids of $\Delta$-matroids.
\begin{theorem}
Let $D$ be a delta matroid with ground set $E$, with upper matroid $M_u$ and lower matroid $M_l$.
Then every cycle in $M_u$ is a union of cycles in $M_l$.
\end{theorem}

{\em Proof.} Let $C$ be a cycle in $M_u$, and consider the restriction %$D' = D \mid C = D \setminus (E \setminus C)$
of $D$ to
$C$, see \cite{bouchet4}.
In this restriction, the upper matroid is a single cycle $C$.
If the lower matroid is not a union of lower cycles, then there is an edge
$e$ which is contained in no lower cycle, so $C \setminus \{e\}$
which is an upper basis, hence feasible, is not lower spanning,
a contradiction.
\hfill $\Box$

It is easy to show that this necessary condition is stronger than the earlier nec-
essary conditions, since if every cycle in $M_u$ is a union of cycles in $M_l$ then every basis of $M_l$ is independent in $M_u$, and every basis of $M_u$ is spanning in $M_l$. In fact, we have necessary and sufficient conditions for two matroids to be the upper and lower matroids of a
$\Delta$-matroid.

%
%
%Nevertheless, we have the following partial converse.
%
%\begin{theorem}
%Let $M_u$ and $M_l$ be matroids on the same ground set $E$ such that every basis of
%$M_l$ is independent in $M_u$, and every basis of $M_u$ is spanning in $M_l$.
%
%Suppose also that every cycle in $M_u$ is a union of cycles in $M_l$.
%
%Then setting $$\mathcal{F} = \{F \subseteq E \mid \mbox{$F$ is spanning in $M_u$}, \mbox{$F$ is spanning in $M_u$}\} $$
%we have that $\mathcal{D}$ is the set of feasible sets of a $\Delta$-matroid.
%\end{theorem}
%
%\begin{proof}

\begin{theorem}
Let $M_u$ and $M_l$ be matroids on the same ground set $E$.  Then $M_u$ and $M_l$ are the upper and lower matroids
of a $\Delta$-matroid if and only if %the following conditions;
%\begin{enumerate}
  % \item every basis of $M_l$ is independent in $M_u$,
  % \item every basis of $M_u$ is spanning in $M_l$, and
  % \item 
every cycle in $M_u$ is a union of cycles in $M_l$.
%\end{enumerate}
%are satisfied.
\end{theorem}

{\em Proof.} We only have to show sufficiency. Suppose that every cycle in $M_u$ is a union of cycles in $M_l$, so that every basis of $M_l$ is independent in $M_u$, and every basis of $M_u$ is spanning in $M_l$. We will construct a $\Delta$-matroid realizing $M_u$ and $M_l$ as the upper and lower matroids. Set $\mathcal{F}$ to be the collection of all subsets of $E$ which are both upper independent and lower spanning. So, in particular, the bases of $M_l$ and $M_u$ are contained in $\mathcal{F}$.
%We have shown above that the three conditions are necessary.  Suppose that 1 through 3 are satisfied.  We will construct
%a $\Delta$ matroid realizing $M_u$ and $M_l$ as the upper and lower matroids.
%Set $\mathcal{F}$ to be the collection of all subsets of $E$ which are both upper independent and lower spanning.
%So, as we have seen, the bases of $M_l$ and $M_u$ are contained in $\mathcal{F}$.

Let $F, F' \in \mathcal{F}$ and let $x \in F \setminus F'$.  If $x$ is an element of a lower-cycle in $F$, then
chose $y = x$, and $F-x$ is both lower-spanning and upper independent, so feasible.
So suppose that $x$ belongs to no lower-cycle in $F$.
Then $F-x$ is not lower spanning, and, since $F'$ is lower spanning, there exists and element $y \in F' /F$ so that
$F - x + y$ is lower spanning.  Moreover, $y$ is contained in no lower cycle in $F - x + y$.
We know that $F$ is upper independent, and so $F-x+y$ is either upper independent, or contains a unique upper-cycle,
which contains $y$.  That upper-cycle is a union of lower-cycles, contradicting the fact that $y$ is  contained in no lower cycle in $F - x + y$.
So $F-x+y$ is lower-spanning and upper-independent.

Now we must consider the case when $x \in F' \setminus F$.
If $F + x$ is upper-independent,  we can choose $y = x$, and
$F \bigtriangleup \{x,y\} = F + \{x\}$ is upper-independent and lower-spanning, as required.
So suppose that $F + x$ is not upper-independent, and so has a unique upper-cycle $C$ containing
$x$.  Since $C \not \subseteq F'$, there exists an element $y \in F\setminus F'$ so that
$F + x - y$ is upper independent.  Moreover, since $y$ is contained in a lower-cycle $C'$ contained in the
lower-spanning set $F + x$, $F+ x - y$ is also lower spanning.

So $\mathcal{F}$ is the collection of feasible sets of a $\Delta$-matroid on $E$.
\hfill $\Box$
%%%%%%%%%%%%%%%%%%

Consider a $\Delta$-matroid $D$ on a ground set $E$. From the symmetric exchange axiom it is clear that replacing the feasible sets $F$ with their complements $E \setminus F$ , yields another delta matroid $D^{\star}$ with feasible sets $\mathcal{F}^\star = \{E \setminus F | F \in \mathcal{F}\}$. Note
that $(M_l)^\star = (M^\star)_u$ and $(M_u)^\star = (M^\star)_l$ . For a subset $X$ of the ground set which is contained in some feasible set $F$, we define a $\Delta$-matroid $D \setminus X$, on the ground set $E \setminus X$, whose feasible sets are $\{F \setminus X | F \in \mathcal{F}\}$ and a $\Delta$-matroid $D/X = (D^\star \setminus X)^\star$.

Oxley \cite{oxley} defines a matroid $Q$ to be a quotient of a matroid $M$ if there is a matroid $N$ and a subset $X$ of the ground set of $N$ such that $M = N \setminus X$ and $Q = N/X$ and proves that $Q$ is a quotient of $M$ if and only if every circuit of $M$ is a union of circuits of $Q$, so we have the following.

\begin{corollary}\label{coro:quo}
Let $M_u$ and $M_l$ be matroids on the same ground set $E$. Then $M_u$
and $M_l$ are the upper and lower matroids of a $\Delta$-matroid if and only if $M_l$ is a quotient of $M_u$.
\end{corollary}

There may be many matroids $N$ on $E \cup X$ with the property that $M_l = N \setminus X$ and $M_u = N/X$. For our $\Delta$-matroid example on the edge set of a graph $G = (V, E)$ with upper matroid the $2$-dimensional generic rigidity matroid of $G$ and lower matroid the cycle matroid of $G$, Corollary \ref{coro:quo} implies that the connectivity matroid of a graph must be a quotient of its $2$-dimensional generic rigidity matroid. Geometrically, this does not seem obvious, however the following is a realization of this relationship. Given $G = (V, E)$, the cone of $G$, $G_c = (V\cup \{x_0\}, E\cup X)$ is obtained by adding a new vertex $x_0$ and $|V |$ new edges connecting the $x_0$ to each
vertex in $V$, so $X = \{(x, v) | v \in V \}$. See Figure \ref{fig:cone}.

\begin{figure}[htb]
\centering
\includegraphics{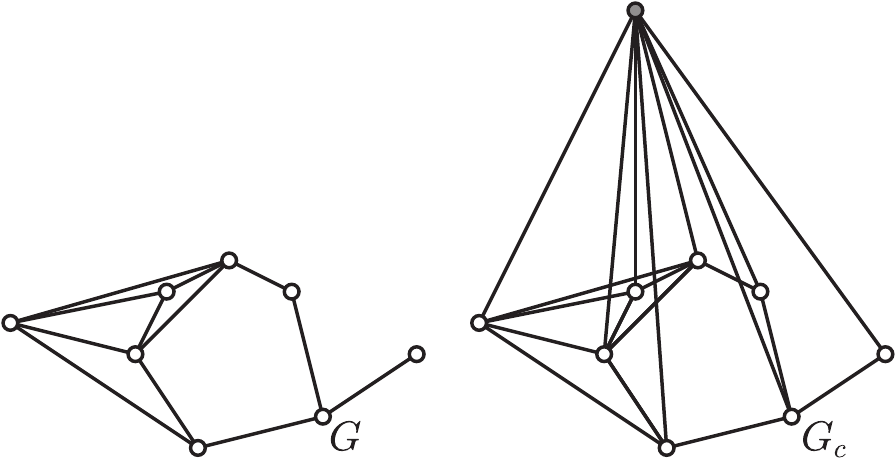}
\caption{A graph $G$ and its cone $G_c$. \label{fig:cone}}
\end{figure}

\begin{theorem}
Given a graph $G$, then its $2$-dimensional generic rigidity matroid
is $M_r(G) = M_r (G_c) \setminus X$ and its connectivity matroid is $M_c(G) = M_r (G_c)/X$.
\end{theorem}

\begin{proof}
Clearly $M_r(G) = M_r(G_c) \setminus X$, so we need only show $M_c (G) = M_r (G_c )/X$. The circuits of $M_r(G _c)/X$ are the minimal sets in
$$
D=\big\{C\setminus N | C\in\mathcal{C}(M_r(G_c))\big\}.
$$
Since the cone of every connectivity cycle is a circuit in the $2$-dimensional generic rigidity matroid, the minimal sets in $D$ can contain at most the edges of a single connectivity cycle.

On the other hand, since the cone of an acyclic graph is independent in the
$2$-dimensional generic rigidity matroid, every edge set in $D$ contains at least one connectivity cycle.

So the minimal elements of D are exactly the connectivity cycles of the graph
$G = (V, E)$, and $M_c(G) = M_r (G_c )/X$.
\end{proof}

If we are given the upper and lower matroids $M_l$ and $M_u$  of a $\Delta$-matroid $D$, we may wonder if there is a way to construct a $\Delta$-matroid $D_{min}$ (resp $D_{max}$)  with the least (resp maximum) number of feasibles and the same upper and lower matroids as $D$. 

Here we consider the special cases where the upper or lower matroid is the uniform matroid. Let $\mathcal{F}_{max}=\mathcal{B}_l\cup\mathcal{I}_{max}(D)$,
$\mathcal{F}'_{max}=\mathcal{B}_l\cup\mathcal{I}'_{max}(D)$ with
\begin{eqnarray*}
\mathcal{I}_{max}(D)&=&\{A\subseteq E\, |\,  \mbox{rk}(M_l)<|A|\leq\mbox{rk}(M_u), \, \cr && \mbox{ $A$ is spanning in $M_l$}\};\cr
\mathcal{I'}_{max}(D)&=&\{A\subseteq E\, |\,  \mbox{rk}(M_l)\leq|A|<\mbox{rk}(M_u), \, \cr && \mbox{ $A$ is independent in $M_u$}\}.
\end{eqnarray*}

\begin{theorem} Let $D=(E, \mathcal{F}(D))$ be a $\Delta$-matroid.
\begin{itemize}
\item If the upper matroid of $D$ is uniform, $D_{max}=(E, \mathcal{F}_{max})$ is the $\Delta$-matroid with maximum number of feasibles satisfying $M_l(D_{max})=M_l$.
\item If the lower matroid of $D$ is uniform,  $D'_{max}=(E, \mathcal{F}'_{max})$ is the $\Delta$-matroid with maximum number of feasibles satisfying $M_u(D'_{max})=M_u$.
\end{itemize}
\end{theorem}
\begin{proof}
We will only prove the first assertion of this theorem because duality can recover the second.

Let $F_1,  F_2\in \mathcal{F}_{max}$, $x\in F_1\Delta F_2$ and let us search for $y\in F_1\Delta F_2$ such that the ($\Delta$F) axiom holds.  There is nothing to prove if $F_1, F_2\in \mathcal{B}_l$; $F_1, F_2\in \mathcal{B}_u$ or $F_1\in \mathcal{B}_l, F_2\in \mathcal{B}_u$  since the ($\Delta$F) axiom is satisfied on elements of the set $\mathcal{F}(D)$.

Assume that $F_1\in \mathcal{B}_l$, $F_2\in \mathcal{I}_{max}(D)\setminus \mathcal{B}_u$ and $x\in F_1\setminus F_2$. From Theorem \ref{uplow},  there exists $K^2_l\subset E$ and $F^2_l\in \mathcal{B}_l$ such that $F_2=F^2_l\cup K^2_l$. As a result, $x\in F_1\setminus F^2_l$ will yield an $x'\in F^2_l\setminus F_1$ such that $F_1-x+x'\in \mathcal{B}_l$ according to the (MB) axiom and then $x'\in F_2\setminus F_1$ allows to take $y=x'$. Since $|F_2|<\mbox{rk}(M_u)$, $F_2+x\in  \mathcal{I}_{max}(D)$,  we can take $y=x\in F_1\Delta F_2$ such that $F_2\Delta \{x\}\in  \mathcal{I}_{max}(D)$. Assume we begin by picking $x\in F_2\setminus F_1$. We can clearly take $y=x$ to get $F_1+x\in  \mathcal{I}_{max}(D)$. We can still use $y=x$  in case $x\in K^2_l$ to obtain $F_2-x=F^2_l\cup (K^2_l-x)\in  \mathcal{I}_{max}(D)$. If $x\in F^2_l$ then $x\in F^2_l\setminus F_1$ which according to the (MB) axiom will give an $x'\in F_1\setminus F^2_l$ such that $F^2_l-x+x'\in \mathcal{B}_l$. If $x'\notin K^2_l$, we can simply take $y=x'$ because $x'\in F_1\setminus F_2$ and $F_2-x+x'=(F^2_l-x+x')\cup K^2_l\in \mathcal{I}_{max}(D)$. If $x'\in K^2_l$, $F_2-x=(F^2_l-x+x')\cup (K^2_l-x')\in \mathcal{I}_{max}(D)$ and we can choose $y=x$.

Assume that  $F_1, F_2\in \mathcal{I}_{max}(D)$. Applying again Theorem \ref{uplow},  we can write $F_1=F^1_l\cup K^1_l, F_2=F^2_l\cup K^2_l$;  for $F^1_l, F^2_l\in \mathcal{B}_l$ and $K^1_l, K^2_l$ subsets of $E$. 

If $x\in K^1_l$, we can take $y=x$ since $F^1_l\Delta \{x\}=F\setminus x=F^1_l\cup(K^1_l\setminus x)\in \mathcal{I}_{max}(D)$. Otherwise, $x\in F^1_l\setminus F_2$ and this implies that $x\in F^1_l\setminus F^2_l$ which by the (MB) axiom gives $x'\in  F^2_l\setminus F^1_l$ such that $F^1_l-x+x'\in \mathcal{B}_l$. In the case $x'\in K^1_l$, we can take $y=x$ as $F_1-x=(F^1_l-x+x')\cup K^1_l\in \mathcal{I}_{max}(D)$ and if $x'\notin K^1_l$  then $x'\in F_2\setminus F_1$ and we take $y=x'$. Furthermore if $|F_2|<\mbox{rk}(M_u)$ than we can take $y=x$ as $|F_2\cup\{x\}|\leq \mbox{rk}(M_u)$ and then $F_2\Delta\{x\}\in \mathcal{I}_{max}(D)$. 

Assume that $|F_2|=\mbox{rk}(M_u)$. In this case, $F_2\in \mathcal{B}_u$, and because $|F_1|\leq\mbox{rk}(M_u)$ there is a subset $K$ of $E$ such that $F_1\cup K\in \mathcal{B}_u$ and $x\in (F_1\cup K)\setminus F_2$ implies that there is $x'\in F_2\setminus (F_1\cup K)$ such that $F_2-x'+x\in \mathcal{B}_u$ from the (MB) axiom. We have just found $y=x'\in F_2\setminus F_1$ such that $F_2-x'+x\in \mathcal{B}_u$.
\end{proof}

We now use an example to illustrate the previous theorem. Let $E=\{a, b, c, d\}$
and $\mathcal{F}$ be the subset of the power set of $E$ containing $M_u=\{\{a, b, c, d\}\}$ and  all the one element and two element subsets of $E$. Clearly $D=(E, \mathcal{F})$ is a $\Delta$-matroid. Adding to $\mathcal{F}$ all the three elements subsets of $E$ gives the maximal $\Delta$-matroid having the same lower matroid and upper matroid with $D$.

We now wonder what happens if we remove the requirement that the lower and upper matroids be uniform. This is a question that will be investigated further in the future. Furthermore, the question of how to find $D_{min}$ remains unanswered. 
If the matroids $D_{min}$ and $D_{max}$ exist, are they uniquely defined?

\vspace{0.5cm}

\end{document}